\documentclass[letterpaper]{article}

\usepackage{spconf,amsmath,graphicx}
\usepackage{soul}
\usepackage{newclude}
\usepackage{color,xcolor}
\usepackage{graphicx}
\usepackage{multirow}
\usepackage{booktabs}
\usepackage{amssymb}
\usepackage{algorithmic}
\usepackage{algorithm}
\usepackage{amsthm}
\usepackage{tabularx}
\usepackage{caption}
\newtheorem{theorem}{Theorem}
\newtheorem{proposition}{Proposition}
\newtheorem{definition}{Definition}

\newtheorem{lemma}{Lemma}
\newtheorem{corollary}{Corollary}
\newtheorem{fact1}{Fact}
\title{R$\acute{\text E}$nyi Differential Privacy in the Shuffle Model: Enhanced Amplification Bounds}
\name{E Chen$^{1,*}$ \quad Yang Cao$^{2,*}$ \thanks{*Correspondence:chene@zhejianglab.com, 
yang@ist.hokudai.ac.jp} \quad Yifei Ge$^{3}$}

\address{$^{1}$ Zhejiang Lab \quad
$^{2}$ Hokkaido University \quad 
$^{3}$ Xi'an Jiaotong-Liverpool University}
\newcommand{\e}{$\acute{\text e}$}
\begin{document}
%
\maketitle
\begin{abstract}
The shuffle model of Differential Privacy (DP) has gained significant attention in privacy-preserving data analysis due to its remarkable tradeoff between privacy and utility. 
It is characterized by adding a shuffling procedure after each user's locally differentially private perturbation, which leads to a  \textit{privacy amplification} effect, meaning that the privacy guarantee of a small level of noise, say $\epsilon_0$, can be enhanced to $O(\epsilon_0/\sqrt{n})$ (the smaller, the more private) after shuffling all $n$ users' perturbed data. 
Most studies in the shuffle DP focus on proving a tighter privacy guarantee of privacy amplification.
However, the current results assume that the local privacy budget $\epsilon_0$ is within a limited range.
In addition, there remains a gap between the tightest lower bound and the known upper bound of the privacy amplification.
In this work, we push forward the state-of-the-art by making the following contributions. 
Firstly, we present the first asymptotically optimal analysis of R\e nyi Differential Privacy (RDP) in the shuffle model without constraints on $\epsilon_0$. 
Secondly, we introduce hypothesis testing for privacy amplification through shuffling, offering a distinct analysis technique and a tighter upper bound.
Furthermore, we propose a DP-SGD algorithm based on RDP. 
Experiments demonstrate that our approach outperforms existing methods significantly at the same privacy level.
 \renewcommand{\thefootnote}{}
\footnote{
This work is partially support by JST CREST JPMJCR21M2, JSPS KAKENHI Grant Number JP22H00521, JP22H03595, JP21K19767, JST/NSF Joint Research SICORP JPMJSC2107.
}

\begin{keywords}
R\e nyi Differential Privacy, Shuffle Model, Privacy Amplification, Hypothesis Testing
\end{keywords}
\end{abstract}
\section{Introduction}\label{sec-1}
R\e nyi Differential Privacy (RDP) \cite{mironov2017renyi,abadi2016deep} is a variant of Differential Privacy (DP) \cite{dwork2014algorithmic,dwork2016calibrating}, which has become the preferred and more sophisticated framework for ensuring privacy protection. While DP is widely regarded as the standard framework for privacy protection, RDP offers enhanced privacy guarantees and finer control over privacy levels by introducing the R\e nyi order parameter. Although advanced composition theorems exist for DP \cite{dwork2014algorithmic} to quantify privacy leakage, it should be noted that these bounds may not always be tight.

The interest in the shuffle model has been driven by its privacy amplification effect \cite{balle2019privacy,cheu2019distributed}. This effect becomes notably pronounced because there is often a trade-off between utility and individual privacy protection through Local Differential Privacy (LDP) \cite{xiong2020comprehensive}.
The primary focus in this area is to achieve more precise RDP bounds for privacy amplification \cite{erlingsson2019amplification,girgis2021renyi,feldman2022hiding,feldman2023stronger}.  
 Figure \ref{Fig:shufflePlot} illustrates the shuffle model, where the choice between adaptive and non-adaptive approaches is determined by the query's dependency on previous queries.

However, it is worth noting that the existing results not only impose limitations on the local privacy budget $\epsilon_0$ but also leave a considerable gap to reach the lower bound of RDP. This suggests the need for further advancements and exploration in this area to bridge the existing gap.

In this work, we provide the first asymptotically optimal analysis for RDP in the shuffle model without any additional constraints on $\epsilon_0$. By transforming the original problem into a simpler task of measuring the distance between two multinomial distributions \cite{feldman2022hiding}, we improve the RDP bound using hypothesis testing \cite{kairouz2015composition,dong2022gaussian}. To provide a clear and convenient comparison, we outline both previous findings and our own results in Table \ref{Tab:RDPComparison}. Despite the advancements and continuous improvements in RDP bounds under the shuffle model, previous studies have imposed constraints on the privacy budget $\epsilon_0$ (such as $0<\epsilon_0<1/2$ \cite{erlingsson2019amplification}).   In contrast, our method eliminates such restrictions and is suitable for all $\epsilon_0>0$.  Furthermore, we introduce a DP-SGD algorithm based on RDP, leveraging the RDP bounds obtained from our asymptotically optimal analysis. The numerical results demonstrate that our proposed algorithm significantly surpasses existing methods in terms of both preserving privacy and maintaining utility. 

\begin{figure}[htbp]
	\centering
		\includegraphics[scale=0.2]{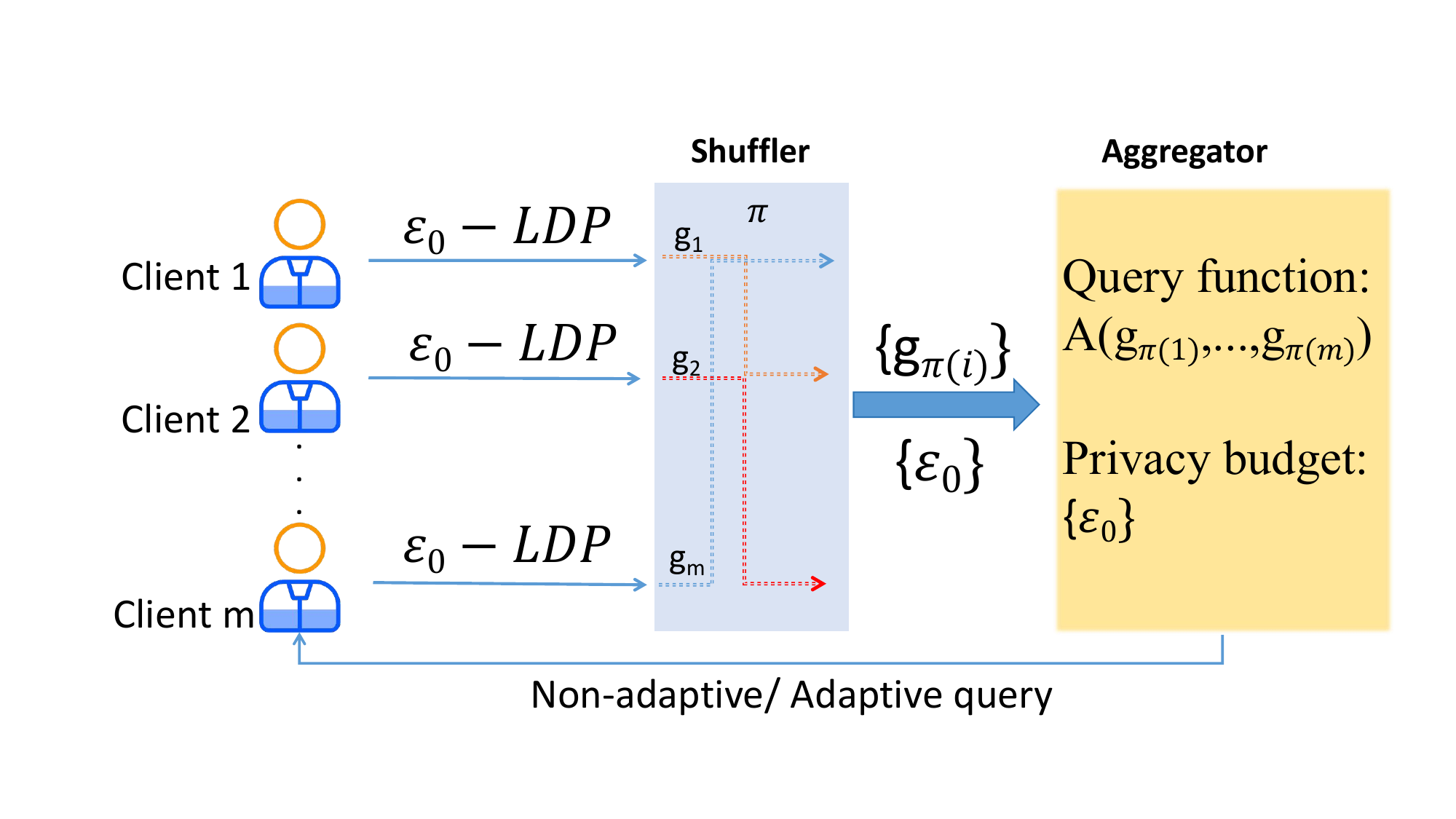}
		\caption{The shuffle model with $\epsilon_0$-LDP users}
  \label{Fig:shufflePlot}
\end{figure}

\begin{table}[htbp]
  \renewcommand{\arraystretch}{1.5}
  \centering
  \small
  
    \begin{tabular}{|l|l|c|}
    \hline
    \multicolumn{1}{|l|}{Methods} & \multicolumn{1}{p{4.11111116409302em}|}{Upper \newline{}Bound } & \multicolumn{1}{p{4.11111116409302em}|}{Lower \newline{}Bound \cite{girgis2021renyi}} \\
    \hline
    Erlingsson et al. (SODA, 2019) \cite{erlingsson2019amplification} & $O( \frac{e^{6\epsilon_0}\lambda}{n})$     & \multirow{4}{*}{\centering $ O(\frac{e^{\epsilon_0} \lambda}{n})$} \\
\cline{1-2}    Girgis et al. (CCS, 2021) \cite{girgis2021renyi} & $O( \frac{e^{2\epsilon_0}\lambda}{n})$     &  \\
\cline{1-2}Feldman et al. (FOCS, 2022) \cite{feldman2022hiding}  & \multicolumn{1}{c|}{\multirow{2}{*}{\centering $O( \frac{64e^{\epsilon_0} \lambda}{n})$}} &  \\
\cline{1-1} Feldman et al. (SODA, 2023) \cite{feldman2023stronger} &       &  \\
\cline{1-2}    This work & $O( \frac{2e^{\epsilon_0}\lambda}{n})$     &  \\
    \hline
    \end{tabular}%
  \caption{Privacy amplification of RDP via shuffling with order $\lambda$. For the sake of brevity, we have not listed the constraints on $\epsilon_0$ found in previous studies.}
\label{Tab:RDPComparison}%
\end{table}%

\section{Preliminaries and Notations}
This section gives essential terminology, definitions and properties related to differential privacy. Due to space limitations, proofs will be included in the appendix.
\begin{definition}
[Pure Differential Privacy]
A randomized algorithm $\mathcal{R}:\mathcal{D} \rightarrow \mathcal{S}$ satisfies 
$\epsilon$-DP
if for all neighboring datasets $D_0, D_1 \in \mathcal{D}^n$, \begin{equation}\label{epsilonDeltaDis}
\mathbb{P}( \mathcal{R}(D_0)\in S) \leq e^\epsilon\mathbb{P}(\mathcal{R}(D_1) \in S).
\end{equation}
Here $D_0$ and $D_1$ are considered neighboring if they differ by exactly one record. If the randomized mechanism is localized, it is called Local Differential Privacy (LDP).
\end{definition}

\begin{definition}\label{RDP-DF}($R\acute{e}nyi$ Divergence)
For two random variables $U$ and $V$, the $R\acute{e}nyi$ divergence of $U$ and $V$ of
order $\lambda>1$ is defined as:
\begin{equation}
D^{\lambda}(U\| V)=\frac{1}{\lambda-1}\log \mathop{\mathbb{E}}\limits_{x \sim V}\left[\left(\frac{U(x)}{V(x)}\right)^\lambda\right].
\end{equation}
\end{definition}

Introduced in \cite{mironov2017renyi}, R\e nyi differential privacy (RDP) can be defined based on  R\e nyi divergence.
\begin{definition}[R\e nyi Differential Privacy]
A randomized mechanism $\mathcal{R}$ is said to be $(\lambda,\epsilon(\lambda))$-RDP
if for all neighbouring pairs $D_0 \sim D_1$, it holds that
  $D^\lambda(\mathcal{R}(D_0)\| \mathcal{R}(D_1)) \le \epsilon.$
\end{definition}
Finally, we establish the framework of privacy protection algorithm under consideration in this paper. The notation $[n]$ represents the set of natural numbers from $1$ to $n$.
For a domain $\mathcal{D}$, let $\mathcal{R}^{(i)}:\mathcal{S}^{(1)}\times \mathcal{S}^{(2)} \times \cdots \times \mathcal{S}^{(i-1)} \times \mathcal{D} \rightarrow \mathcal{S}^{(i)}$ ($i$ $\in [n]$) be a sequence of algorithms such that $\mathcal{R}^{(i)}(z_{1:i-1},\cdot)$ is an $\epsilon_0$-LDP randomizer for all values of auxiliary inputs $z_{1:i-1} \in \mathcal{S}^{(1)} \times \mathcal{S}^{(2)} \times \cdots \times \mathcal{S}^{(i-1)}$ where $\mathcal{S}^{(i)}$ is the range space of $\mathcal{R}^{(i)}$. Let $\mathcal{A}_R: \mathcal{D} \rightarrow \mathcal{S}^{(1)} \times \mathcal{S}^{(2)} \times \cdots \times \mathcal{S}^{(n)}$ represent the algorithm applied to the given dataset $x_{1:n} \in \mathcal{D}^n$. The algorithm sequentially computes $z_i = \mathcal{R}^{(i)}(z_{1:i-1},x_i)$ for $i \in [n]$ and outputs $z_{1:n}$.
 We refer to $\mathcal{A}_R(\mathcal{D})$ as an $\epsilon_0$-LDP adaptive process. Alternatively, if we first uniformly sample a permutation $\pi: [n] \rightarrow [n]$, and then sequentially compute $z_i = \mathcal{R}^{(i)}(z_{1:i-1},x_{\pi_i})$ for $i \in [n]$, we say it is a shuffled process and denote it as $\mathcal{A}_{R,S}(\mathcal{D})$.
 Here, $\pi_i = \pi(i)$ represents the position of $i$ after permutation.

For the sake of brevity and convenience of notation, we omit $D$ and use $\mathcal{A}_{R}$, $\mathcal{A}_{R,S}$ to represent the
adaptive process and the shuffled adaptive process, respectively.

\begin{proposition}\label{Pro:DP2Multi}
(Feldman et al. \cite{feldman2022hiding})
For a domain $\mathcal{D}$, let $\mathcal{A}_R$ be the $\epsilon_0$-LDP adaptive process and $\mathcal{A}_{R,S}$ be the related shuffled $\epsilon_0$-LDP adaptive process.
  Assume $X_0$ = $(x^0_1,x_2,\ldots,x_n)$ and $X_1 = (x^1_1,x_2,\ldots,x_n)$ be two neighbouring datasets such that for all $j \neq 1, x_j \notin \{x^0_1,x^1_1\}$. Suppose that there exists a positive value $p \in (0,1]$ such that for all $i \in [n], x \in \mathcal{D} \backslash \{x^0_1,x^1_1\}$ and $z_{1:i-1} \in \mathcal{S}^{(1)} \times \mathcal{S}^{(2)} \times \cdots \times \mathcal{S}^{(i-1)}$, there exists a distribution $LO^{(i)}(z_{1:i-1},x)$ such that
\begin{equation}
\begin{split}
\mathcal{R}^{(i)}(z_{1:i-1},x)  =& \frac{p}{2}\mathcal{R}^{(i)}(z_{1:i-1},x^0_1)+\frac{p}{2}\mathcal{R}^{(i)}(z_{1:i-1},x^1_1) \\
&+
(1-p) LO^{(i)}(z_{1:i-1},x).
\end{split}
\end{equation}
Then there exists a randomized postprocessing algorithm $f$ such that $\mathcal{A}_s(X_0)$ is distributed identically to $f(A+\Delta,C-A+1-\Delta)$ and $\mathcal{A}_{R,S}(X_1)$ is distributed identically to $f(A+1-\Delta,C-A+\Delta)$, where $p=e^{-\epsilon_0}, \Delta \sim Bern(\frac{e^{\epsilon_0}}{e^{\epsilon_0}+1}), C \sim Bin(n-1,p), A \sim Bin(C,1/2)$.
\end{proposition}
The starting point of this paper is Proposition \ref{Pro:DP2Multi}, which transforms the original problem into a simpler task of analyzing a non-adaptive protocol.  It is worth noting that Proposition \ref{Pro:DP2Multi} mentions the joint distribution of $A$ and $C$, which corresponds to the multinomial distribution $Multinom(n-1;p/2,p/2,1-p)$. Here, $A$ and $C-A$ represent the number of 0s and 1s, respectively.
\section{Privacy amplification by shuffling based on multinomial distribution }\label{sec-2}

\subsection{The Exact RDP Bound  for the Shuffle Model}
In order to provide our tighter exact closed-form bound, we employ a combination of maximum likelihood and hypothesis testing methods, leveraging their strengths to derive a more accurate and precise result.

\begin{theorem}\label{thm::Exact}
Let $P=(A,C-A+\Delta)$ and $Q=(A+\Delta,C-A)$, where $p=e^{-\epsilon_0}, C \sim Bin(n-1,p), A \sim Bin(C,1/2), \Delta \sim Bern(\frac{1}{1+p})$.
Then  $$
D^\lambda(P\|Q) = \frac{1}{\lambda-1}
\log \int_0^1 |h'(x)|^{1-\lambda}dx.
$$
Here, $h(\alpha)$ can be obtained as follows.
\begin{equation*}\label{eqn:f}
\small
\begin{aligned}
h(\alpha) = 1-\alpha-c_0\sum\limits_{v=0}^{n-1}\mathbb{P}\left(B_1(v)< A \le B_2(v) |C=v \right)\mathbb{P}(C=v),
\end{aligned}
\end{equation*}
where
{\small $$\alpha(t) = \mathbb{P}(\Delta=0)\mathbb{P}(\frac{A}{C-A+1}>t)+\mathbb{P}(\Delta=1)\mathbb{P}(\frac{A+1}{C-A}>t),$$}
  $g(\alpha) = \inf \limits_t \{t:\alpha(t) \le \alpha\},$ $c_0 = \frac{e^{-\epsilon_0}-1}{e^{-\epsilon_0}+1}$, $B_1(v) = \frac{g(\alpha)v-1}{g(\alpha)+1},
B_2(v) = \frac{g(\alpha)v-1}{g(\alpha)+1}+1$, $\mathbb{P}(C=v)=\binom{n-1}{v} e^{-v\epsilon_0}(1-e^{-\epsilon_0})^{n-1-v}, v=0,1,\cdots,n-1$, 
$P(A=k|C=v)=\binom{v}{k}(\frac{1}{2})^k$, $k=0,1,\cdots,v.$
\end{theorem}

Theorem \ref{thm::Exact} gives exact RDP bound for the shuffled output, however, since RDP is symmetric, we also need to provide result of 
$D^\lambda(Q\|P)$ in a similar way. Without causing any ambiguity, we use notations
$P$ and $Q$ to represent the same value in Theorem \ref{thm::Exact} in the following context. In fact, we can obtain the R\e nyi divergence between two multinomial distributions by numerical computation \cite{feldman2022hiding}.

\begin{corollary}\label{coro:ExactRDP}
The $\epsilon_0$-LDP shuffled adaptive process satisfies $(\lambda, \max\{D^\lambda (P\|Q), D^\lambda(Q\|P)\}$-RDP, where $P,Q$ and $D^\lambda (\cdot\|\cdot)$ are defined in Theorem \ref{thm::Exact}.
\end{corollary}

\subsection{The Asymptotic RDP Bound for the Shuffle Model}
Although Corollary \ref{coro:ExactRDP} provides a RDP bound for the shuffle model, it is excessively complicated and lacks intuitive understanding. In the following context, we provide an asymptotic RDP bound. This approach not only simplifies computations but also facilitates the extension to other privacy definitions based on the relationship between divergences and Gaussian Differential Privacy (GDP) \cite{dong2022gaussian}.
\begin{lemma}\label{MultiCLT1}
Assume $\xi$ = $(n_0,n_1,n_2)'$ is a random variable which obeys multinomial distribution with parameters $(n-1;\frac{p}{2},\frac{p}{2},1-p)$, then $\xi$ approximately follows the multivariate normal distribution
$N(\tilde{\pmb{\mu}},\tilde{\pmb{\Sigma}})$ as $n \rightarrow \infty$,  where $\tilde{\pmb{\mu}}=\left(\frac{(n-1)p}{2},\frac{(n-1)p}{2},(n-1)(1-p)\right)'$ and covariance matrix of $\xi$
is
$$\tilde{\pmb{\Sigma}} = (n-1)\left(
               \begin{array}{ccc}
                 \frac{p}{2}(1-\frac{p}{2}) & -\frac{p^2}{4} & -\frac{p(1-p)}{2} \\
                 -\frac{p^2}{4} & \frac{p}{2}(1-\frac{p}{2}) & -\frac{p(1-p)}{2} \\
                 -\frac{p(1-p)}{2} & -\frac{p(1-p)}{2} & p(1-p) \\
               \end{array}
             \right).
$$
\end{lemma}
Lemma \ref{MultiCLT1} provides the asymptotic normality of the multinomial distribution, which is closely related to GDP. The Berry-Esseen type central limit theorem \cite{pinelis2016optimal} ensures that the theoretical convergence rate is at most $O\left(\frac{1}{\sqrt{n}}\right)$ while the numerical calculation confirms a convergence rate of approximately $O\left(\frac{1}{n}\right)$, which gives an upper bound.
\begin{theorem}\label{thm:asym}
For a domain $\mathcal{D}$, the shuffled $\epsilon_0$-LDP adaptive process approximately satisfies 
$\frac{2e^{\epsilon_0/2}}{\sqrt{n-1}}$-GDP.
\end{theorem}

According to the fact that $\mu$-GDP implies $(\lambda, \frac{1}{2}\mu^2\lambda)$-RDP \cite{dong2022gaussian}, we have the asymptotic RDP bound of the shuffled output.
\begin{corollary}\label{coro:aym}
For a domain $\mathcal{D}$, the shuffled $\epsilon_0$-LDP adaptive process approximately satisfies 
$(\lambda, \frac{2e^{\epsilon_0}\lambda}{n-1})$-RDP for any $\lambda \ge 2$.
\end{corollary}

\subsection{Comparison of RDP Bounds under the Shuffle Model}
Bounds on RDP for privacy amplification via shuffling were initially introduced by Erlingsson et al. \cite{erlingsson2019amplification}. 
Girgis et al. \cite{girgis2021renyi} gave both an upper bound and a lower bound of RDP for $\epsilon_0 \ge 0$ and any integer $\lambda \ge 2$. That is, the RDP of the shuffled $\epsilon_0$-LDP adaptive process is upper-bounded by
 \begin{equation}
 \epsilon(\lambda) \le \frac{1}{\lambda-1} \log \left( e^{\lambda^2 \frac{(e^{\epsilon_0}-1)^2}{\bar{n}}}+e^{\epsilon_0 \lambda - \frac{n-1}{8 e^{\epsilon_0}}} \right),
 \end{equation}
 where $\bar{n} = \lfloor \frac{n-1}{2e^{\epsilon_0}} \rfloor +1$. And the RDP of the shuffled $\epsilon_0$-LDP adaptive process is lower-bounded by
 \begin{equation}
 \epsilon(\lambda) \ge \frac{1}{\lambda-1} \log \left( 1+\frac{\lambda(\lambda-1)}{2} \frac{(e^{\epsilon_0}-1)^2}{ne^{\epsilon_0}} \right).
 \end{equation}

The exponential term $e^{\epsilon_0 \lambda - \frac{n-1}{8 e^{\epsilon_0}}}$ in the upper bound comes from the Chernoff bound, it goes to $0$ rapidly as $n$ increases. If we omit this term, the upper bound is nearly$ \frac{2}{n-1} \frac{\lambda}{\lambda-1}e^{\epsilon_0}(e^{\epsilon_0}-1)^2,$
which is worse than our simplified bound in Corollary \ref{coro:aym} by a multiplicative factor of $\frac{\lambda}{\lambda-1}(e^{\epsilon_0}-1)^2$.  
Although the RDP bound has been improved to 
$O(\frac{64 e^{\epsilon_0}\lambda}{n})$ \cite{feldman2023stronger}, our proposed RDP bound shows a significantly superior performance.

\begin{figure}[htbp]
	\centering
	\begin{minipage}{0.48\linewidth}
		\centering
		\includegraphics[width = 1.05\linewidth]{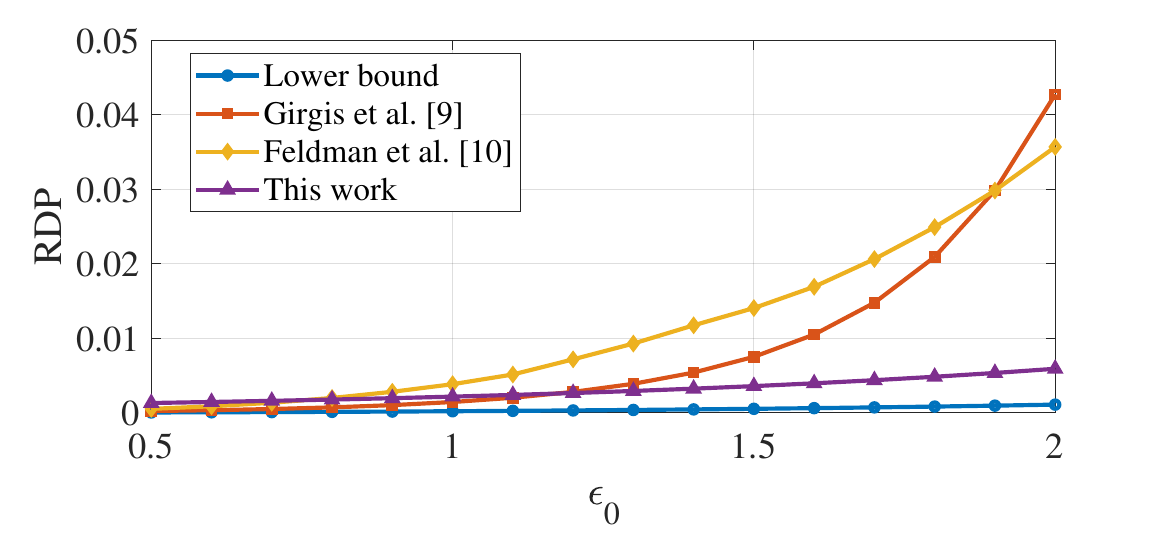}
		\caption{RDP as a function of $\epsilon_0$ for $\lambda=4$ and $n=10^4$}
		\label{Fig:RDP_epsilon}
	\end{minipage}
	\begin{minipage}{0.48\linewidth}
		\centering
		\includegraphics[width = \linewidth]{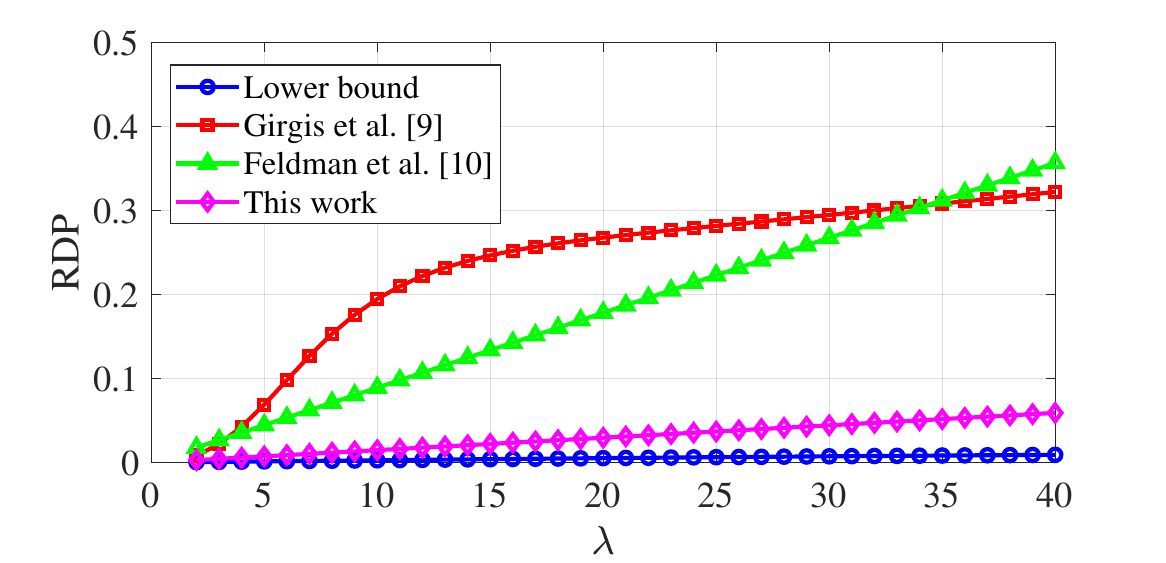}
		\caption{RDP as a function of $\lambda$ for $\epsilon_0=2$ and $n=10^4$}
		\label{Fig:RDP_lambda}%
	\end{minipage}
\end{figure}

Figure \ref{Fig:RDP_epsilon} illustrates that our RDP bound provides a significantly tighter bound compared to the existing work  for a fixed value of $\lambda=4$. Moreover, it demonstrates that for both $\epsilon_0 > 1$ and $\epsilon_0 < 1$, our RDP bound and the lower bound are in close proximity. Similarly, Figure \ref{Fig:RDP_lambda} shows that our RDP bound and the lower bound are close for each $\lambda \geq 2$, with a fixed $\epsilon_0=2$.    
\section{Application and Experiments}\label{sec-application}
Stochastic gradient descent (SGD) is a crucial algorithm for empirical risk minimization (ERM), which aims to minimize a parameterized function given by
$\mathcal{L}(\pmb \theta) = \sum_{i=1}^n \ell(\pmb \theta,x_i),$
where $\pmb \theta \in \mathbb{R}^{d}$. 
Several studies have focused on a differentially private variant of Stochastic Gradient Descent (SGD) \cite{bassily2014private,song2013stochastic,girgis2021shuffled}. Additionally, researchers have investigated a deep learning version of SGD tailored for the popular MNIST handwritten digit dataset \cite{abadi2016deep}.

In our experiments, we consider a scenario with $m$ clients, where each client has $n/m$ samples. Our approach, which incorporates Principle Components Analysis (PCA), achieves an impressive accuracy of 98.18\% after approximately 50 epochs in the non-private scenario. This accuracy result is consistent with the findings of a standard neural network  trained on the same MNIST dataset \cite{lecun1998gradient}. By employing this methodology, we can effectively train a simple classifier that achieves high accuracy in recognizing handwritten digits from the MNIST dataset. Details of the experiment setup can be found in Table \ref{tab:Experimentsetup}.

To the best of our knowledge, the best RDP bound for the shuffled noisy SGD with $\epsilon_0$-LDP adaptive process is listed in \cite{girgis2021renyi,feldman2022hiding}, while Figure \ref{Fig:RDP_epsilon} and \ref{Fig:RDP_lambda} show that the privacy bound in this work is tighter.
Furthermore, our technique is based on Laplace mechanism  and can be applied to stochastic gradient descent with batch size $m$. Figure \ref{Fig:DPSGD} demonstrates that our method achieves the closest accuracy (96.41\%) to the true result (98.18\%) when the same total privacy budget (RDP) is applied.
This result is natural because a better RDP bound allows for more precise local computations, leading to more accurate results.
\begin{table}[htbp]
  \centering
  \small
  \caption{Experiment setting for the shuffled SGD on the MNIST dataset}
    \begin{tabular}{lll}
    \toprule
    \textbf{Parameters/Setting \quad} & \multicolumn{1}{l}{\textbf{Value}} & \textbf{Explanation} \\
    \midrule
    Activation function & ReLU & \\
    Output layer   & Softmax & \\
    Loss function & Cross-entropy & \\
   Input layer & 60 variables & 60 PCA components \\ 
    $C$     & $10$    & Clipping bound \\
    $\epsilon_0$ & \multicolumn{1}{l}{$[0.1,2]$} & Privacy budget \\
    $\eta$ & $0.05$   & Step size \\
    $m$     & $100$   & Batch size \\
    $n$    & $60,000$ & Sample size \\
    $T$    & $50$     & Epoch count \\
    \bottomrule
    \end{tabular}%
  \label{tab:Experimentsetup}%
\end{table}%

\begin{proposition}[Composition theorem of GDP \cite{dong2022gaussian}]\label{pro:GDPcomposition}
 The $k$-fold composition of $\mu_i$-GDP mechanisms is $\sqrt{\mu_1^2 + \cdots +
\mu_k^2}$.   
\end{proposition}
\begin{proposition}[Laplace mechanism \cite{dwork2014algorithmic}]\label{LaplacePrivacy}
For a function $g: D \rightarrow \mathbb{R}^d$, let $l_1$ sensitivity be defined as
$\Delta(g) = \max_{D_0 \sim D_1} \|g(D_0)-g(D_1)\|_1$,
then for any $\epsilon_0 >0$, the noisy output $h(D) = g(D)+ Lap(\Delta(g)/\epsilon)$ satisfies $\epsilon_0$-LDP.
\end{proposition}
\begin{theorem}\label{ProSGD1}
Algorithm \ref{alg:SGD} approximately satisfies 
$(\lambda,\frac{2Te^{\epsilon_0}\lambda}{m-1})$-RDP.

\end{theorem}
\begin{proof}
In each epoch, the algorithm is consisted of two main steps splitting and shuffling, let
\begin{align*}
&\mathcal{R}^{(i)}(z_{1:i-1},D_{\pi(i)}) \\
=& \tilde{\pmb \theta}_i=\tilde{\pmb \theta}_{i-1}(z_{1:i-1})-\eta_i (\nabla \ell(\tilde{\pmb \theta}_{i-1}(z_{1:i-1}),D_{\pi(i)})+\pmb b_i) ,
\end{align*}
then the output of Algorithm \ref{alg:SGD} can be seen as post processing of the shuffled $m$ blocks. Since $l_1$ sensitivity of each $\mathcal{R}^{(i)}(z_{1:i-1},\cdot)$ is  $\frac{2C}{m}$, then it is $\epsilon_0$-LDP  according to Proposition \ref{LaplacePrivacy}. Combined with Theorem \ref{thm:asym} and  Proposition \ref{pro:GDPcomposition}, Algorithm \ref{alg:SGD} approximately satisfies $\frac{2\sqrt{T}}{\sqrt{m-1}} e^{\frac{\epsilon_0}{2}}$-GDP. Since  $\mu$-GDP implies $(\lambda, \frac{1}{2}\mu^2\lambda)$-RDP, the proof is completed.
\end{proof}
\begin{algorithm}[htbp]
    \caption{The Shuffled noisy SGD for $\epsilon_0$-LDP users}
    \label{alg:SGD}
    \begin{algorithmic}[1]
        \REQUIRE $X = (x_1,\ldots,x_n)$, $\mathcal{L}(\pmb \theta,x)$, $\pmb \theta_0$, $\eta$, $T$, $\epsilon_0$, $d$, $m$, $C$
        \ENSURE $\hat{\pmb \theta}_{T,m}$   
        \STATE Split $[n]$ into $m$ disjoint subsets $S_1, \cdots, S_m$ with equal size $n/m$
        \STATE Choose arbitrary initial point $\hat{\pmb \theta}_{0,m}$
        \FOR{each $t \in [T]$}
            \STATE $\tilde{\pmb\theta}_0 = \hat{\pmb \theta}_{t-1,m}$
            \STATE Choose a random permutation $\pi$ of $[m]$
            \FOR{each $i \in [m]$}
                \STATE $\pmb b_i \sim Lap(0,\frac{2C}{\epsilon_0 m} \pmb{I}_d)$
                \FOR{each $j \in S_{\pi(i)}$}
                    \STATE $\pmb g_i^{j} = \nabla \ell(\tilde{\pmb \theta}_{i-1},x_j)$
                    \STATE $\tilde{\pmb g}_i^j = \pmb g_i^j/\max(1,\| \pmb g_i^j \|_1/C)$
                \ENDFOR
                \STATE $\tilde{\pmb \theta}_{i} = \tilde{\pmb \theta}_{i-1} - \eta (\frac{1}{m} \sum_j \tilde{\pmb g}_i^j+\pmb b_i)$
            \ENDFOR
            \STATE $\hat{\pmb \theta}_{t,m} = \tilde{\pmb \theta}_m$
        \ENDFOR
        \RETURN $\hat{\pmb \theta}_{T,m}$
    \end{algorithmic}
\end{algorithm}

\begin{figure}[htbp]
	\centering
		\includegraphics[scale=0.3]{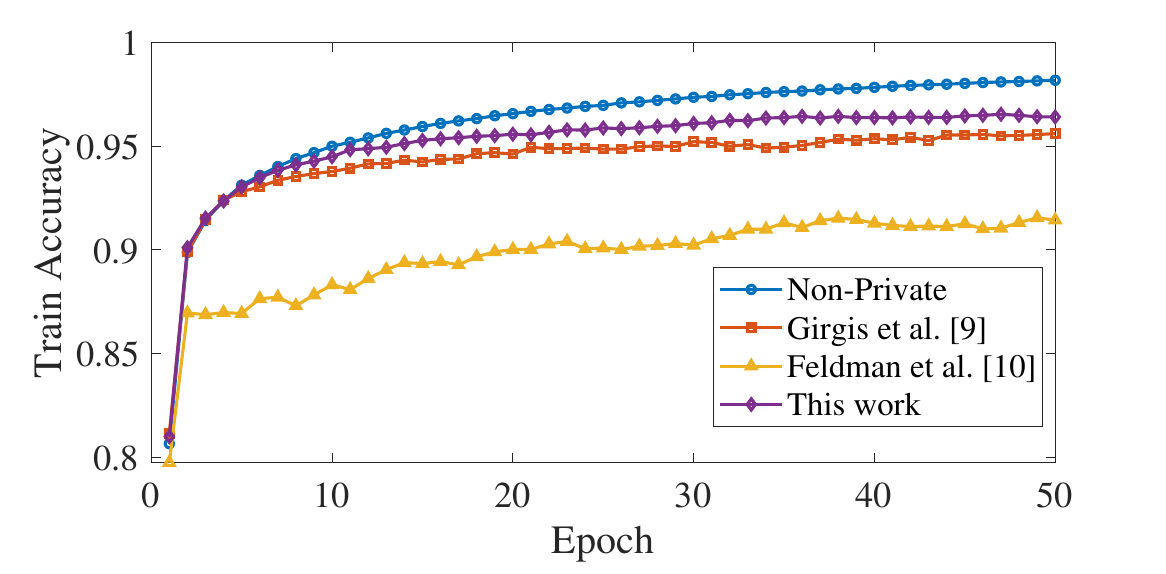}
		\caption{Comparison of train accuracy at the same privacy level of $(\lambda,0.008\lambda$)-RDP. Non-Private refers to the standard SGD without adding noise. }
		\label{Fig:DPSGD}
\end{figure}
\section{Conclusion}
In conclusion, this study focused on R\e nyi Differential Privacy in the shuffle model for privacy protection. We provide an analysis of RDP without any constraints on the privacy budget for the first time. Furthermore, a more efficient DP-SGD algorithm based on RDP was proposed, outperforming existing methods within the same privacy budget. 
\newpage
\bibliographystyle{IEEEbib}
\bibliography{refs}
\newpage
\setcounter{section}{0} 
\renewcommand{\thesection}{\Alph{section}}
\section{appendix}
\setcounter{theorem}{0} 
\setcounter{lemma}{0} 
\setcounter{definition}{0} 
\setcounter{fact}{0} 
\renewcommand{\thetheorem}{\Alph{section}.\arabic{theorem}} 
\renewcommand{\thelemma}{\Alph{section}.\arabic{lemma}} 
\renewcommand{\thedefinition}{\Alph{section}.\arabic{definition}}
\renewcommand{\thefact}{\Alph{section}.\arabic{fact}}
First, we introduce a powerful tool called $f$-DP   based on hypothesis testing . Readers can refer to the reference \cite{dong2022gaussian} for more detailed information. 

Let $U$ and $V$ denote the probability distributions of $M(D_0)$ and $M(D_1)$, respectively. We consider a rejection rule $0 \leq \phi \leq 1$, with type I and type II error rates defined as
\begin{equation}\label{TypeError}
\alpha_\phi = \mathbb{E}_U[\phi], \quad \beta_{\phi} = 1-\mathbb{E}_V[\phi].
\end{equation}
\begin{definition}(Trade-off function)
For any two probability distributions $U$ and $V$ on the same space $\Omega$, the trade-off function $T(U,V):[0,1]\rightarrow [0,1]$ is defined as
$$
T(U,V)(\alpha) = \inf\{\beta_\phi: \alpha_\phi \leq \alpha \},
$$
where the infimum is taken over all measurable rejection rules $\phi$, and $\alpha_{\phi}=\mathbb{E}_U(\phi)$ and $\beta_{\phi}=1-\mathbb{E}_V(\phi)$.
\end{definition}
\begin{definition}($f$-differential privacy, $f$-DP)
Let $f$ be a trade-off function, a mechanism $M$ is said to be $f$-differentially private if
\begin{equation}\label{TF}
T(M(D_0),M(D_1))(\alpha) \ge f(\alpha),
\end{equation}
for all neighboring data sets $D_0$ and $D_1$ and $0 \le \alpha \le 1$. Under the assumption of no ambiguity, we assume that the trade-off function is a function of $\alpha$ and denote the abbreviated form of equation (\ref{TF}) as $T(M(D_0),M(D_1)) \ge f$.
\end{definition}

\begin{lemma} \cite{dong2022gaussian}\label{f2Renyi}
    If a mechanism $M$ is $f$-DP and $\inf \{x\in [0,1]:f(x)=0 \}=1$, then it is $(\lambda, \ell_\lambda^{R\acute{e}nyi}(f))$-RDP for any $\lambda>1$.
    Here, $\ell_\lambda^{R\acute{e}nyi}(f) = \frac{1}{\lambda-1}\log \int^1_0 |f'(x)|^{1-\lambda} dx$. 
\end{lemma}


Here are some useful facts about $f$-DP and details can be found in \cite{dong2022gaussian} .


\begin{fact1}\label{fact:postprocess}
If a mechanism $M$ is $f$-DP, then its post-processing $Proc \circ M$ is also $f$-DP.
\end{fact1}
\begin{fact1}\label{muexpand}
A $f$-DP mechanism is called $\mu$-GDP if $f$ can be obtained by $f=T(N(0,1),N(\mu,1))=\Phi(\Phi^{-1}(1-\alpha)-\mu)$. In other words,
A mechanism is $\mu$-GDP if and only if it is $(\epsilon,\delta(\epsilon))$-DP for all $\epsilon \ge 0$, where $$
\delta(\epsilon)=\Phi(-\frac{\epsilon}{\mu}+\frac{\mu}{2})
-e^{\epsilon}\Phi(-\frac{\epsilon}{\mu}-\frac{\mu}{2}),
$$
where $\Phi(\cdot)$ is cumulative distribution function of standard normal distribution $N(0,1)$.
\end{fact1}
\begin{fact1}\label{kfold-comp}
The $k$-fold composition of $\mu_i$-GDP mechanisms is $\sqrt{\mu_1^2 + \cdots +
\mu_k^2}$.
\end{fact1}

\begin{lemma}\label{Compdist}
Given observation $(a,b)$, the likelihood function between $H_0: P=(A,C-A+1)$ and $H_1: Q = (A+1,C-A)$ is $\Lambda = \frac{\mathcal{L}_{H_1}}{\mathcal{L}_{H_0}} = \frac{\mathbb{P}(Q=(a,b))}{\mathbb{P}(P=(a,b))} =
\frac{a}{b}$, where $C \sim Bin(n-1,p), A \sim Bin(C,1/2)$.
\end{lemma}
\begin{proof}
Since $C \sim Bin(n-1,p)$, we can obtain that $\mathbb{P}(C=k)=\binom{n-1}{k}p^k(1-p)^{n-1-k}, k=0,1,\cdots,n-1$. Then
$\mathbb{P}(Q=(a,b))=\mathbb{P}(C=a+b-1) \mathbb{P}(A=a-1|C=a+b-1) = \mathbb{P}(C=a+b-1) \mathbb{P}(A=b|C=a+b-1)\\
 = \frac{a}{b}\mathbb{P}(P=(a,b)),$
 which indicates that $\Lambda = \frac{\mathcal{L}_{H_1}}{\mathcal{L}_{H_0}}=\frac{a}{b}$.
\end{proof}

\begin{lemma}\label{Neyman}(Neyman-Pearson lemma)
For a hypothesis test between null hypothesis $H_0$:The sample distribution is $P$ and alternative hypothesis $H_1$:The sample distribution is $Q$, let $\Lambda = \frac{\mathcal{L}_{H_1}}{\mathcal{L}_{H_0}}$, then
the optimal reject region for $\omega \in \Omega$(sample space) is defined as follows:
\begin{equation}\label{Neyman}
\phi(\omega)=\left\{
\begin{aligned}
1 & , & \Lambda > t, \\
0 & , & \Lambda \le t.
\end{aligned}
\right.
\end{equation}
\end{lemma}

\begin{lemma}\label{TNorm}
 Assume $\pmb{\mu}_0$,$\pmb{\mu}_1 \in \mathbb{R}^d$, $\pmb{I}_d \in \mathbb{R}^{d\times d}$ is an identity matrix and $\pmb{\Sigma} \in \mathbb{R}^{d \times d}$ is a positive definite matrix, then
 $$T(N(\pmb{0},\pmb{I}_d),N(\pmb{\mu},\pmb{I}_d)) = \Phi(\Phi^{-1}(1-\alpha)-\sqrt{\pmb{\mu}'\pmb{\mu}}), $$
 and
  \begin{align*}
& T(N(\pmb{\mu}_0,\pmb{\Sigma}),N(\pmb{\mu}_1,\pmb{\Sigma}))  \\
& = \Phi(\Phi^{-1}(1-\alpha)-\sqrt{(\pmb{\mu}_1-\pmb{\mu}_0)'\pmb{\Sigma}^{-1}(\pmb{\mu}_1-\pmb{\mu}_0)}),
\end{align*}   where $\alpha \in [0,1]$ represents type I error.
\end{lemma}
\begin{proof}
 Notice that likelihood function between $H_0: N(\pmb{0},\pmb{I}_d)$ and $H_1: N(\pmb{\mu},\pmb{I}_d)$ is
\begin{equation}
\frac{L_{H_1}}{L_{H_0}}=\frac{\prod\limits_{j=1}^d(2\pi)^{\frac{d}{2}}e^{-\frac{1}{2}(\pmb{x}-\pmb{\mu})'(x-\pmb{\mu})}}{\prod\limits_{j=1}^d(2\pi)^{\frac{d}{2}}e^{-\frac{1}{2}\pmb{x}'\pmb{x}}}
=e^{-\frac{1}{2}\pmb{\mu}'\pmb{\mu}+\pmb{\mu}'\pmb{x}},
\end{equation}
according to Lemma A.3, the reject rule is defined as follows:
\begin{equation}
\phi(\omega)=\left\{
\begin{aligned}
1 & , & \pmb{\mu}'\pmb{x} > t, \\
0 & , & \pmb{\mu}'\pmb{x} \le t.
\end{aligned}
\right.
\end{equation}
Hence, type I error
$\alpha(t)=\mathbb{P}(\pmb{\mu}'\pmb{x}>t)$, where $x \sim N(0,\pmb{I}_d)$, which indicates that $\pmb{\mu'}\pmb{x} \sim N(0,\pmb{\mu}'\pmb{\mu})$.
Therefore,
\begin{equation}
\alpha(t)=\mathbb{P}(\frac{\pmb{\mu}'\pmb{x}}{\sqrt{\pmb{\mu}'\pmb{\mu}}}>\frac{t}{\sqrt{\pmb{\mu}'\pmb{\mu}}}) = 1- \Phi(\frac{t}{\sqrt{\pmb{\mu}'\pmb{\mu}}}),
\end{equation}
which indicates that $t = \sqrt{\pmb{\mu}'\pmb{\mu}}\Phi^{-1}(1-\alpha)$. In addition, we can calculate type II error
\begin{equation}\label{eqn:f}
\begin{aligned}
\beta(t) & = \mathbb{P}(\pmb{\mu}'(\pmb{x}+\pmb{\mu})\le t)  \\
& = \mathbb{P}(\frac{\pmb{\mu}'\pmb{x}}{\sqrt{{\pmb{\mu}'\pmb{\mu}}}}\le \frac{t}{\sqrt{\pmb{\mu}'\pmb{\mu}}}-\sqrt{\pmb{\mu}'\pmb{\mu}}), \\
& = \Phi(\frac{t}{\sqrt{\pmb{\mu}'\pmb{\mu}}}-\sqrt{\pmb{\mu}'\pmb{\mu}})=\Phi(\Phi^{-1}(1-\alpha)-\sqrt{\pmb{\mu}'\pmb{\mu}}).
\end{aligned}
\end{equation}
Based on simple calculation, we can find that
\begin{align*}
& T(N(\pmb{\mu}_0,\pmb{\Sigma}),N(\pmb{\mu}_1,\pmb{\Sigma}))  \\
& = T(N(\pmb{0},\pmb{I}_d),N(\pmb{\Sigma}^{-1/2}(\pmb{\mu}_1-\pmb{\mu}_0),\pmb{I}_d)  \\
& = \Phi(\Phi^{-1}(1-\alpha)-\sqrt{(\pmb{\mu}_1-\pmb{\mu}_0)'\pmb{\Sigma}^{-1}(\pmb{\mu}_1-\pmb{\mu}_0)}).
\end{align*}   
\end{proof}
When $d=1$, we have $T(N(0,1),N(\mu,1)) = \Phi(\Phi^{-1}(1-\alpha)-\mu)$, which is consistent with the conclusion in Fact \ref{muexpand}.

\subsection{Proof of Theorem 1}
\begin{lemma}\label{lem:nodelta}
Let $P=(A,C-A+1)$ and $Q=(A,C-A)$, where $p=e^{-\epsilon_0}, C \sim Bin(n-1,p), A \sim Bin(C,1/2),$
then  
$$T(P,Q) \ge h.$$
Here, $\alpha(t), \beta(t)$ and $h(\alpha)$ can be obtained as follows.
\begin{equation*}
\alpha(t) = \sum \limits_{v=0}^{n-1}\mathbb{P}(A>\frac{tv+t}{tv-1})\mathbb{P}(C=v),
\end{equation*}
\begin{equation*}
g(\alpha) = \inf \limits_t \{t:\alpha(t) \le \alpha\},
\end{equation*}
\begin{equation*}
\beta(t)=h(\alpha) = 1-\alpha-\sum\limits_{v=1}^{n-1}\mathbb{P}(A = \lceil\frac{g(\alpha)v-1}{g(\alpha)+1}\rceil)\mathbb{P}(C=v).
\end{equation*}
\end{lemma}
\begin{proof}
Consider the reject rule in Lemma A.3, then
type I error \\

\begin{equation*}\label{eqn:1}
\begin{aligned}
\alpha(t)& = \mathbb{P}(\frac{A}{C-A+1}>t)  \\
& =\sum \limits_{v=0}^{n-1}\mathbb{P}(A>\frac{tv+t}{t+1})\mathbb{P}(C=v),
\end{aligned}
\end{equation*}
and type II error
\begin{equation*}\label{eqn:2}
\small
\begin{aligned}
\beta(t) & = \mathbb{P}(\frac{A+1}{C-A} \le t)  \\
& = \sum \limits_{v=0}^{n-1}\mathbb{P}(A \le \frac{tv-1}{t+1}) \mathbb{P}(C=v) \\
& = 1-\alpha(t)-\sum\limits_{v=0}^{n-1}\mathbb{P}(\frac{tv-1}{t+1}<A \le \frac{tv-1}{t+1}+1) \mathbb{P}(C=v).
\end{aligned}
\end{equation*}
Since $\alpha(t)$ is non-increasing, then for each $\alpha$, there exists unique $t$ such that $t=g(\alpha)$, where $g(\alpha)$ = $\mathop{\inf}\limits_t$ $\{t:\alpha(t) \le \alpha\}$. Hence, $\beta(t)=h(\alpha) = 1-\alpha-\sum\limits_{v=0}^{n-1}\mathbb{P}(A = \lceil\frac{g(\alpha)v-1}{g(\alpha)+1}\rceil)\mathbb{P}(C=v)$.
\end{proof}
\subsubsection*{Proof of Theorem \ref{thm::Exact}}
\begin{proof}
Similar in proof of Lemma \ref{lem:nodelta}, we can obtain that
\begin{equation*}\label{eqn:3}
\small
\begin{aligned}
\alpha(t) & = \mathbb{P}(\frac{A+\Delta}{C-A+1-\Delta}>t)  \\
& = \mathbb{P}(\Delta=0)\mathbb{P}(\frac{A}{C-A+1}>t)+\mathbb{P}(\Delta=1)\mathbb{P}(\frac{A+1}{C-A}>t),
\end{aligned}
\end{equation*}

\begin{equation*}\label{eqn:4}
\small
\begin{aligned}
\beta(t) & = \mathbb{P}(\frac{A+1-\Delta}{C-A+\Delta}\le t)  \\
& = \mathbb{P}(\Delta=0)\mathbb{P}(\frac{A+1}{C-A} \le t) + \mathbb{P}(\Delta=1)\mathbb{P}(\frac{A}{C-A+1} \le t).
\end{aligned}
\end{equation*}
Let $g(\alpha)$ = $\inf\limits_t$ $\{t:\alpha(t) \le \alpha\}$ and substitute the probability function of the $A$ and $C$ into the expression, we can obtain that
\begin{equation}\label{eqn:f}
\small
\begin{aligned}
\beta(t) & = h(\alpha)  \\
& = 1-\alpha-c_0\sum\limits_{v=1}^{n-1}\mathbb{P}\left(B_1(v)< A \le B_2(v) |C=v \right)\mathbb{P}(C=v),
\end{aligned}
\end{equation}
where $c_0 = \frac{e^{-\epsilon_0}-1}{e^{-\epsilon_0}+1}$, $B_1(v) = \frac{g(\alpha)v-1}{g(\alpha)+1},
B_2(v) = \frac{g(\alpha)v-1}{g(\alpha)+1}+1$, $\mathbb{P}(C=v)=\binom{n-1}{v} e^{-v\epsilon_0}(1-e^{-\epsilon_0})^{n-1-v}, v=0,1,\cdots,n-1$, 
$P(A=k|C=v)=\binom{v}{k}(\frac{1}{2})^k$, $k=0,1,\cdots,v$.
Then it follows that the mechanism $M$ satisfies $h$-DP. By combining this  with Lemma \ref{f2Renyi},
the proof is completed.
\end{proof}

\subsection{Proof of Lemma \ref{MultiCLT1}}
Consider the following two multinomial distributions $P_1$ and $Q_1$, where $P_1:= (n_0+1, n_1, n_2)$, $Q_1:= (n_0, n_1+1, n_2)$ and $(n_0, n_1, n_2) \sim MultiNom(n-1;\frac{p}{2},\frac{p}{2},1-p)$, $p=e^{-\epsilon_0}$.
It is easy to check that $T(P,Q) = T(P_1,Q_1)$, where $P$ and $Q$ are distributions defined in Lemma \ref{Compdist}. In addition, if $P_2$ and $Q_2$ are distributions defined in
Theorem \ref{thm::Exact}, then we need to calculate $T(P_2,Q_2)$, where $P_2 = (n_0+1-\Delta,n_1+\Delta,n_2)$ and
$Q_2 = (n_0+\Delta,n_1+1-\Delta,n_2)$,  $\Delta \sim Bern(p)$.
The following lemma is a basic property of multinomial distribution.

\begin{lemma}\label{ProMulti}
Assume $\xi$ = $(n_0,n_1,n_2)'$ is a random variable which obeys $MultiNom(n-1;\frac{p}{2},\frac{p}{2},1-p)$, then expectation of $\xi$ is
$\tilde{\pmb{\mu}}=(\frac{(n-1)p}{2},\frac{(n-1)p}{2},(n-1)(1-p))'$ and covariance matrix of $\xi$
is
$$\tilde{\pmb{\Sigma}} = (n-1)\left(
               \begin{array}{ccc}
                 \frac{p}{2}(1-\frac{p}{2}) & -\frac{p^2}{4} & -\frac{p(1-p)}{2} \\
                 -\frac{p^2}{4} & \frac{p}{2}(1-\frac{p}{2}) & -\frac{p(1-p)}{2} \\
                 -\frac{p(1-p)}{2} & -\frac{p(1-p)}{2} & p(1-p) \\
               \end{array}
             \right)
$$
\end{lemma}
This lemma can be obtained from simple calculation directly. Next, we prove that the multinomial distribution approximately obeys the multivariate normal distribution.

\subsubsection*{Proof of Lemma \ref{MultiCLT1}}
\begin{proof}
Since $\xi=(n_0,n_1,n_2)'$ obeys $MultiNom(n-1;p_0,p_1,p_2)$ with $p_0=p/2, p_1=p/2$ and $p_2 = 1-p$, the characteristic function
\begin{equation}
\phi_\xi(t_0,t_1,t_2) = E(e^{it_0n_0+it_1n_1+it_2n_2}),
\end{equation}
then
\begin{equation}\label{eqn:f}
\small
\begin{aligned}
\phi_\xi(t_0,t_1,t_2) & = \sum\limits^{n-1}_{n_0,n_1,n_2=0}\frac{n!}{n_0!n_1!n_2!}
p_0^{n_0}p_1^{n_1}p_2^{n_2}e^{it_0n_0+it_1n_1+it_2n_2}\\
& = (p_0e^{it_0}+p_1e^{it_1}+p_2e^{it_2})^{n-1}.
\end{aligned}
\end{equation}
Let $\tilde{\xi} = (\frac{n_0-(n-1)p_0}{\sqrt{n-1}},\frac{n_1-(n-1)p_1}{\sqrt{n-1}},\frac{n_2-(n-1)p_2}{\sqrt{n-1}})'$, then we have
$\phi_{\tilde{\xi}}(t_0,t_1,t_2) =$
$$(p_0e^{\frac{it_0}{\sqrt{n-1}}}+ p_1e^{\frac{it_1}{\sqrt{n-1}}}+p_2e^{\frac{it_2}{\sqrt{n-1}}})^{n-1}\\
e^{-\sqrt{n-1}i(p_0t_0+p_1t_1+p_2t_2)}.$$
Based on Taylor expansion of $e^y = 1+y+\frac{y^2}{2}+o(y^2)$ near point $0$, we can derive that
\begin{equation*}
\begin{aligned}
 & p_0e^{\frac{it_0}{\sqrt{n-1}}}+p_1e^{\frac{it_1}{\sqrt{n-1}}}+p_2e^{\frac{it_2}{\sqrt{n-1}}}  \\
=& 1+i\frac{a}{\sqrt{n-1}}-\frac{b}{2(n-1)}+o(\frac{1}{n-1}),
\end{aligned}
\end{equation*}
where $a=\sum\limits_{j=0}^2 p_j t_j, b=\sum\limits_{j=0}^2 p_j t_j^2.$ In addition, since $\log(1+y)=1-\frac{y^2}{2}+o(y^2)$ when $y\rightarrow 0$, we can obtain that
$ln(\psi_{\tilde \xi}(t))\rightarrow-\frac{b-a^2}{2}$ and
\begin{equation*}
\begin{aligned}
 -\frac{b-a^2}{2}&= -\frac{1}{2}[\sum\limits_{j=0}^2 p_j(1-p_j)t_j^2-\sum\limits_{j=0,j\ne k}^2 \sum\limits_{k=0}^{2} p_jp_kt_it_k]  \\
&= -\frac{1}{2}t'\tilde{\pmb{\Sigma}} t,
\end{aligned}
\end{equation*}
which is the logarithm of the characteristic function of multivariate normal distribution $N(\pmb{0},\tilde{\pmb{\Sigma}})$,
where $\tilde{\pmb{\Sigma}}$ is defined in Lemma  \ref{ProMulti}. According to the uniqueness of the characteristic function, $\tilde \xi \sim N(\pmb{0},\tilde{\pmb{\Sigma}})$ when $n$ is sufficiently large, which indicates that $\xi$ approximately follows $ N(\tilde{\pmb{\mu}},\tilde{\pmb{\Sigma}})$.
\end{proof}

\subsection{Proof of Theorem \ref{thm:asym}}
\begin{proof}
According to Proposition \ref{Pro:DP2Multi}, the key component is to measure distance between two distributions $P_2$ and $Q_2$, where $P_2 = (n_0+1-\Delta,n_1+\Delta,n_2)$ and
$Q_2 = (n_0+\Delta,n_1+1-\Delta)$, $(n_0, n_1, n_2) \sim MultiNom(n-1;\frac{p}{2},\frac{p}{2},1-p)$, $\Delta \sim Bern(p)$ and $p = e^{-\epsilon_0}$.
Let $P = (n_0+1,n_1,n_2)$ and
$Q = (n_0,n_1+1,n_2)$, we first prove that
\begin{equation}\label{eqP1Q1}
T(P,Q)=T(Q,P)\ge \Phi\left(\Phi^{-1}(1-\alpha)-\frac{2}{\sqrt{n-1}}e^{\frac{\epsilon_0}{2}}\right)
\end{equation}
approximately.
If the inequality (\ref{eqP1Q1}) holds, then $$T(P_2,Q_2)=\mathbb{P}(\Delta=0)T(P,Q)+\mathbb{P}(\Delta=1)T(Q,P),$$
 which indicates that
 \begin{equation}
 T(P_2,Q_2)=(\mathbb{P}(\Delta=0)+\mathbb{P}(\Delta=1))T(P,Q)=T(P,Q)
\end{equation}
 by using the fact that $T(P,Q) = T(Q,P)$.\\
We will now prove that the inequality (\ref{eqP1Q1}) holds.
Combined with  Lemma \ref{MultiCLT1},
the key is to prove the following formula,
\begin{equation}
T(N(\tilde{\pmb{\mu}}_0,\tilde{\pmb{\Sigma}}),N(\tilde{\pmb{\mu}}_1,\tilde{\pmb{\Sigma}}))\ge \Phi\left(\Phi^{-1}(1-\alpha)-\frac{2}{\sqrt{n-1}}e^{\frac{\epsilon_0}{2}}\right),
 \end{equation}
 where $\tilde{\pmb{\mu}}_0 = \mathbb{E}(P), \tilde{\pmb{\mu}}_1 = \mathbb{E}(Q)$,  and $\tilde{\pmb{\Sigma}}=Cov(P)=Cov(Q)$.
 However, $\tilde{\pmb{\Sigma}}$ is not irreversible because there exists a linear relationship between sub-vectors: $n_0+n_1+n_2=n-1$. Hence, we need to combine postprocessing properties of $f$-DP and the fact that sub-vector of  a multivariate normal variable is still a multivariate normal variable.\\
Since $P = (n_0+1,n_1,n_2)'$ and $Q = (n_0,n_1+1,n_2)'$, we can use sub-vector $P_2 = (n_0+1,n_1)'$ and $Q_2 = (n_0,n_1+1)'$ to construct $P$ and $Q$ by formula $n_2=n-1-n_0-n_1$. Then, $P=Proc(P_2)$ and $Q=Proc(Q_2)$, which indicates that $T(P,Q)\le T(P_2,Q_2)$ by Fact \ref{fact:postprocess}.\\
From another perspective, it is obvious that $P_2=Proc(P)$ and $Q=Proc(Q)$ if we take $Proc(\cdot)$ as a projection, then $T(P,Q) \ge T(P_2,Q_2)$.
Notice that $P_2$ obeys $N(\pmb{\mu}_0,\pmb{\Sigma})$ and $Q_2$ obeys $N(\pmb{\mu}_1,\pmb{\Sigma})$,
where $\pmb{\mu}_0=(\frac{(n-1)p}{2}+1,\frac{(n-1)p}{2})', \pmb{\mu}_1=(\frac{(n-1)p}{2},\frac{(n-1)p}{2}+1)'$ and $$\pmb{\Sigma} = (n-1) \left(                                                                  \begin{array}{cc}
                                                                                                                       \frac{p}{2}(1-\frac{p}{2}) & -\frac{p^2}{4} \\
                                                                                                                       -\frac{p^2}{4} & \frac{p}{2}(1-\frac{p}{2}) \\
                                                                                                                     \end{array}
         \right).$$
After simple algebraic operation, we can obtain that $$\pmb{\Sigma}^{-1} = \frac{1}{n-1}\left(  \begin{array}{cc}
  \frac{2-p}{p(1-p)} & \frac{1}{1-p} \\
  \frac{1}{1-p} & \frac{2-p}{p(1-p)}\\
\end{array}
 \right),$$
and
\begin{equation}\label{SqrtBound}
\begin{aligned}
& (\pmb{\mu}_1-\pmb{\mu}_0)'\pmb{\Sigma}^{-1} (\pmb{\mu}_1-\pmb{\mu}_0) \\
&= (-1,1)\frac{1}{n-1}\left(\begin{array}{cc}
\frac{2-p}{p(1-p)} & \frac{1}{1-p} \\
\frac{1}{1-p} & \frac{2-p}{p(1-p)}\\
\end{array}
\right)\left(
\begin{array}{c}
-1 \\
1 \\
\end{array}
\right)  \\
& = \frac{4}{(n-1)p}.
\end{aligned}
\end{equation}
Then $T(P,Q) = \Phi(\Phi^{-1}(1-\alpha)-\frac{2}{\sqrt{(n-1)p}})$, and the inequality (\ref{eqP1Q1}) can be obtained immediately by using Lemma \ref{TNorm}.
\end{proof}

\end{document}